\documentclass[11pt,a4paper]{amsart}
\usepackage{amssymb,amsmath,epsfig,graphics,mathrsfs,enumerate,verbatim}
\usepackage[pagebackref,colorlinks=true,linkcolor=blue,citecolor=blue]{hyperref}
\usepackage{fancyhdr}
\usepackage{hyperref}
\hypersetup{
 colorlinks   = true,
 urlcolor     = blue,
 linkcolor    = blue,
 citecolor   = red ,
 bookmarksopen=true
}

\usepackage{amsmath}
\usepackage{amsfonts}
\usepackage{amssymb}
\usepackage{amsthm}
\usepackage{epsfig,graphics,mathrsfs}
\usepackage{graphicx}
\usepackage{dsfont}

\usepackage[usenames, dvipsnames]{color}

\usepackage{hyperref}

\def \phi {\varphi}

\def \RN {\mathbb{R}^N}
\def \R {\mathbb{R}}

\newcommand{\Rn}{\mathbb R^n}
\newcommand{\Rm}{\mathbb R^m}

\newcommand{\Hn}{\mathbb H^n}

\newcommand{\p}{\partial}

\newcommand{\bG}{\mathbb {G}}

\newcommand{\la}{\lambda}

\numberwithin{equation}{section}

\newcommand{\beq}{\begin{equation}}
\newcommand{\bea}[1]{\begin{array}{#1} }
\newcommand{\eeq}{ \end{equation}}
\newcommand{\ea}{ \end{array}}

\newcommand{\ve}{\varepsilon}

\newcommand{\sul}{\Delta_H}

\newcommand{\sa}{\langle}
\newcommand{\da}{\rangle}

\newcommand{\C}{\mathbb{C}}

\newcommand{\fB}{\mathscr L_{\Phi,\Psi}}
\newcommand{\fBa}{\mathscr L_{\Phi,\Psi,\alpha}}

\newtheorem{theorem}{Theorem}[section]
\newtheorem{lemma}[theorem]{Lemma}
\newtheorem{proposition}[theorem]{Proposition}

\newtheorem{remark}[theorem]{Remark}

\numberwithin{equation}{section}

\begin{document}

\title[Positive solutions,  etc.]{Positive solutions of a critical equation in sub-Finsler geometry}



\date{}

\begin{abstract}
We compute a two-parameter family of explicit positive solutions of a critical Yamabe type equation for a nonlinear operator that sits at the intersection of  Finsler and sub-Riemannian geometry.
\end{abstract}

\author{Nicola Garofalo}

\address{Dipartimento d'Ingegneria Civile e Ambientale (DICEA)\\ Universit\`a di Padova\\ Via Marzolo, 9 - 35131 Padova,  Italy}
\vskip 0.2in
\email{nicola.garofalo@unipd.it}

\author{Paolo Salani}
\address{Department of Mathematics and Computer Science\\
Universit\`a di Firenze\\
50134 Florence,
ITALY}\email[Paolo Salani]{paolo.salani@unifi.it}

\thanks{N. Garofalo is supported in part by a Progetto SID (Investimento Strategico di Dipartimento): ``Aspects of nonlocal operators via fine properties of heat kernels", University of Padova (2022); and by a PRIN (Progetto di Ricerca di Rilevante Interesse Nazionale) (2022): ``Variational and analytical aspects of geometric PDEs". He has also been partially supported by a Visiting Professorship at the Arizona State University. P. Salani is supported in part by a PRIN (2022): "Geometric-Analytic Methods for PDEs and Applications (GAMPA)"}

\maketitle

%\tableofcontents

\section{Introduction}\label{S:intro}

In the recent work \cite{DGGS} the authors have discovered an explicit fundamental solution of a new partial differential equation which sits at the intersection of Finsler and sub-Riemannian geometry. These two areas of investigation both represent an extension of Riemannian geometry, but have so far developed independently of one another. On the other hand, as explained in the above mentioned work, there is an interest in merging them into a unified body. 

To introduce the relevant framework of this note, consider the product space $\RN= \Rm\times \R^k$, with coordinates $z\in \Rm$, $\sigma\in \R^k$. Suppose that on each of the two layers of $\RN$ we assign  Minkowski norms $\Phi$ and $\Psi$. By this we mean that
$\Phi^2\in C^{2}(\Rm\setminus\{0\})$, $\Psi^2\in C^{2}(\R^k\setminus\{0\})$, and are strictly convex. We respectively denote by $\Phi^0$ and $\Psi^0$ their dual norms 
\begin{equation}\label{dual0}
\Phi^0(z) = \underset{\Phi(\zeta) = 1}{\sup}\ \sa z,\zeta\da, \ \ \ \ \ \ \ \ \Psi^0(\sigma) = \underset{\Phi(\eta) = 1}{\sup}\ \sa \sigma,\eta\da.
\end{equation}
The levels sets of the functions $\Phi^0$ and $\Psi^0$ are often referred to as \emph{Wulff shapes}. 
In \cite{DGGS} the following energy functional was introduced
\begin{equation}\label{Fen}
\mathscr E^{\Phi,\Psi}_{\alpha,p}(u) = \frac 1p \int_{\RN} \left(\Phi(\nabla_z u)^2 + \frac{\Phi^0(z)^{2\alpha}}4 \Psi(\nabla_\sigma u)^2\right)^{\frac p2} dzd\sigma,\ \ \ \ \ \ \ \ \ \ \ \alpha>0,\ 1<p<\infty.
\end{equation}
As there illustrated, when $\Phi(z) = |z|$, $\Psi(\sigma) = |\sigma|$, in which case one has $\Phi^0(z) = |z|$, $\Psi^0(\sigma) = |\sigma|$, energies such as \eqref{Fen} arise in the study of problems ranging from analysis and geometry of CR manifolds, to the theory of free boundaries, that of quasiconformal mappings, or to the strong rigidity of locally symmetric spaces. 

There is a family of anisotropic dilations naturally attached to \eqref{Fen},
\begin{equation}\label{dil}
\delta_t(z,\sigma) = (t z, t^{\alpha+1} \sigma),\ \ \ \ \ \ \ t>0,
\end{equation}
with respect to which Lebesgue measure in $\RN$ scales according to the formula 
\[
d(\delta_t(z,\sigma)) = t^Q dzd\sigma,
\]
where we have let
\begin{equation}\label{Q}
Q = Q_\alpha = m + (\alpha+1) k.
\end{equation}
Such number plays the role of a dimension. In this connection, we note that, when $1<p<Q$, the energy \eqref{Fen} is the natural one for the Sobolev type embedding 
\begin{equation}\label{sob}
\left(\int_{\RN} |u(z,\sigma)|^q dzd\sigma\right)^{1/q} \le C(N,\alpha,p)\ \mathscr E^{\Phi,\Psi}_{\alpha,p}(u),\ \ \ \ \ \ \ \frac 1p - \frac 1q = \frac 1Q,
\end{equation}
where $u\in C^\infty_0(\RN)$, and $C(N,p)>0$ is a suitable constant.

When $p=2$, and $\Phi(z) = |z|$, $\Psi(\sigma) = |\sigma|$, the following explicit fundamental solution with pole at $(0,0)\in \RN$ for the Euler-Lagrange equation of \eqref{Fen} was first discovered in \cite{G}
\begin{equation}\label{Ga}
 \Gamma_\alpha(z,\sigma)=\frac{C_{\alpha}}{R_\alpha(z,\sigma)^{Q-2}},
\end{equation}
where $C_\alpha>0$ is an explicit constant and 
\begin{align}\label{ra}
R_\alpha(z,\sigma) & = \left(|z|^{2(\alpha+1)}+4(\alpha+1)^2|\sigma|^2\right)^{\frac1{2(\alpha+1)}}.
\end{align}

The dilations \eqref{dil}, and formulas \eqref{Ga} and \eqref{ra}, suggest that one should attach to the energy \eqref{Fen}  the following anisotropic \emph{Minkowski gauge} in $\RN$ 
\begin{equation}\label{theta}
\Theta(z,\sigma) = \left(\Phi(z)^{2(\alpha+1)} + 4(\alpha+1)^2 \Psi(\sigma)^2\right)^{\frac{1}{2(\alpha+1)}},
\end{equation}
for which one clearly has $\Theta(\delta_t(z,\sigma)) = t \Theta(z,\sigma)$. In \cite{DGGS} a new notion of anisotropic Legendre transform $\Theta^0$ was introduced, namely
\begin{equation}\label{theta0}
\Theta^0(z,\sigma)^{\alpha+1}= \underset{\Theta(\xi,\tau) = 1}{\sup}\ \bigg(|\sa z,\xi\da|^{\alpha+1}+ 4(\alpha+1)^2 \sa\sigma, \tau\da\bigg),
\end{equation}
and the following notable property was established for such transform
\begin{equation}\label{theta0}
\Theta^0(z,\sigma) = \left(\Phi^0(z)^{2(\alpha+1)} + 4(\alpha+1)^2 \Psi^0(\sigma)^2\right)^{\frac{1}{2(\alpha+1)}},
\end{equation}
where $\Phi^0$ and $\Psi^0$ are the usual dual norms of $\Phi$ and $\Psi$, as in \eqref{dual0} above.
In one of the main results in \cite{DGGS} it was proved that, with $Q$ as in \eqref{Q}, the function
\begin{equation}\label{wow}
\mathscr G_{\alpha,p}(z,\sigma) = \begin{cases}
C_{\alpha,p}\ \Theta^0(z,\sigma)^{-\frac{Q-p}{p-1}},\ \ \ \ \ \ \ \ p\not= Q,
\\
\\
C_\alpha\ \log \Theta^0(z,\sigma),\ \ \ \ \ \ \ \ \ \ \ \ p = Q,
\end{cases}
\end{equation}
is a fundamental solution with pole in $(0,0)$ of the Euler-Lagrange equation of \eqref{Fen}. In \eqref{wow}, the positive constants $C_{\alpha,p}$ and $C_\alpha$ are explicitly given, and they involve the Wulff shapes of the gauge $\Theta^0$.

\begin{remark}\label{R:ag}
The reader should note the complete symmetry between $\Theta$ and $\Theta^0$ in \eqref{theta}, \eqref{theta0}. Because of such symmetry, the two gauges are indistinguishable in the special case when $\Phi(z) = |z|$, $\Psi(\sigma) = |\sigma|$. With this in mind, the function \eqref{ra} that appears in the above cited \eqref{Ga} from \cite{G} should in fact be understood as the $\Theta^0$ in \eqref{wow}.   
\end{remark}

In this note we are interested in the special situation in which $p=2$ and $\alpha=1$. In such case, the Euler-Lagrange equation of \eqref{Fen} involves the  following nonlinear operator in $\RN$
\begin{equation}\label{Fel2}
\mathscr L_{\Phi,\Psi}(u) = \Delta_{\Phi}(u) + \frac{\Phi^0(z)^{2}}4  \Delta_\Psi(u),
\end{equation}
where we have respectively indicated with $\Delta_{\Phi}$ and $\Delta_{\Psi}$ the Finsler Laplacians in $\Rm$ and $\R^k$ with respect to the norms $\Phi$ and $\Psi$, i.e., the quasilinear operators 
\begin{equation}\label{Flap}
\Delta_\Phi(u) = \operatorname{div}_z(\Phi(\nabla_z u)\nabla \Phi(\nabla_z u)),\ \ \ \ \ \ \ \ \ \Delta_\Psi(u) = \operatorname{div}_\sigma(\Psi(\nabla_\sigma u)\nabla \Phi(\nabla_\sigma u)),
\end{equation}
When $\Phi(z) = |z|$, $\Psi(\sigma) = |\sigma|$, we have $\Phi^0(z) = |z|$, $\Psi^0(\sigma) = |\sigma|$, and \eqref{Fel2} becomes the linear Baouendi-Grushin operator 
\begin{equation}\label{bg}
\mathscr B(u) = \Delta_z u + \frac{|z|^2}4 \Delta_\sigma u.
\end{equation} 
As it is well-known, \eqref{bg} represents the action on cylindrical functions of the horizontal Laplacian in a group of Heisenberg type $\bG$, 
\begin{equation}\label{sul}
\sul u =  \Delta_z + \frac{|z|^2}{4} \Delta_\sigma  + \sum_{\ell = 1}^k \p_{\sigma_\ell} \sum_{i<j} b^\ell_{ij} (z_i \p_{z_j} -  z_j \p_{z_i}),
\end{equation}
where $b^\ell_{ij}$ indicate the group constants given by $[e_i,e_j] = \sum_{\ell=1}^k b^\ell_{ij} \ve_\ell$, see e.g. \cite[Section 2.5]{Gems}. Here, $z = \sum_{i=1}^m z_i e_i \in \Rm$ represents the variable point in the horizontal layer of the Lie algebra, whereas $\sigma =\sum_{\ell = 1}^k \sigma_\ell \ve_\ell\in \R^k$ indicates the variable point in the center of $\bG$. 

To state the main result in this note, we recall that Jerison and Lee introduced the CR Yamabe problem: \emph{Given a compact, strictly pseudo-convex} CR \emph{manifold, find a choice of contact form for which the Webster-Tanaka pseudo-hermitian scalar curvature is constant}. Such problem was solved in most cases by these authors in a series of important papers, see \cite{JL1}, \cite{JL2}, \cite{JL3} and \cite{JL4}. 
A crucial step in their analysis was the explicit computation of the extremal functions in the case $p=2$ of the Folland-Stein Sobolev embedding in the  Heisenberg group $\Hn$. 
Jerison and Lee made the deep discovery that, up to group translations and dilations, a suitable multiple of the function
\begin{equation}\label{JLminimizer}
u(z,\sigma) = ((1 + |z|^2)^2 + \sigma^2)^{-n/2},
\end{equation}
is the only positive entire solution of the CR Yamabe equation in  $\Hn$
\begin{equation}\label{iwa}
\sul u = - u^{1+\frac{2}{n}},
\end{equation}
where we have denoted with $(z,\sigma)$, $z\in \mathbb{C}^n, \sigma\in \mathbb{R}$, the variable point in $\Hn$. In the subsequent paper \cite{GVduke}, Vassilev and the first named author of this note proved a partial extension of such result. They showed that in a Lie group of Iwasawa type $\bG$, up to left-translations, all nontrivial entire solutions to \eqref{iwa} having partial symmetry must be of the type    
\begin{equation}\label{yamG}
u_\ve(z,\sigma) = \left(\frac{(m+2(k-1))\ve^2}{(\ve^2 + |z|^2)^2 + 16 |\sigma|^2}\right)^{\frac{m+2(k-1)}4},
\end{equation}
for some $\ve>0$.

With all this being said, we consider now the partial differential equation
\begin{equation}\label{yamme}
\mathscr L_{\Phi,\Psi}(u) = - u^{\frac{m+2(k+1)}{m+2(k-1)}},
\end{equation}
where $\fB$ is defined by \eqref{Fel2}. We explicitly note that the exponent $\frac{m+2(k+1)}{m+2(k-1)}$ is critical for the Euler-Lagrange equation of the constrained minimisation problem for \eqref{sob} when $\alpha=1$ and $p=2$. The following is the main result of this note.

\begin{theorem}\label{T:yam}
For $\ve>0$ and $\sigma_0\in \R^k$, consider the function
\begin{equation}\label{yam}
K_{\ve,\sigma_0}(z,\sigma) = \left(\frac{(m+2(k-1))\ve^2}{(\ve^2 + \Phi^0(z)^2)^2 + 16 \Psi^0(\sigma+\sigma_0)^2}\right)^{\frac{m+2(k-1)}4}.
\end{equation}
Then $K_{\ve,\sigma_0}$ is a positive entire solution of the partial differential equation \eqref{yamme} in $\RN$. 
\end{theorem}

We conjecture that the functions in \eqref{yam} are the only nontrivial positive solutions to \eqref{yamme}, and that the best constant in \eqref{sob} is attained on them. We plan to return to this conjecture in a future study. 

The present paper is organised as follows. In Section \ref{S:yam} we prove Theorem \ref{T:yam}. In the appendix in Section \ref{S:app} we have collected some of the basic properties of Minkowski norms that are used in this note.

%%%%%%%%%%%%%%%%%%%%%%%%%%%%%%%%%%%%%%%%%%%%%%%%%%%%%%%%

\section{Proof of Theorem \ref{T:yam}}\label{S:yam}

In this section we prove Theorem \ref{T:yam}. We break the proof into several lemmas. 
We begin with some general considerations.
Consider the operator in $\RN$
\[
\fBa(v) = \Delta_{\Phi}(v) + \frac{\Phi^0(z)^{2\alpha}}4  \Delta_\Psi(v).
\]
We make the ansatz that a positive solution $v$ to the Yamabe type equation
\begin{equation}\label{gummybear}
\fBa(v) = - v^{\frac{m+(\alpha+1)k+2}{m+(\alpha+1)k-2}},
\end{equation}
 is an appropriate multiple of the function $u = \rho^{-(m+(\alpha+1)k-2)}$, where for some $\la>0$ the function $\rho:\RN\to [0,\infty)$ solves the \emph{intertwining equation} 
\begin{equation}\label{gummy}
\fBa(\rho^{-(m+(\alpha+1)k-2)}) = - \la\ \rho^{-(m+(\alpha+1)k+2)}.
\end{equation}
Henceforth, we will make frequent use of the following chain rule, whose elementary proof we leave to the reader.

\begin{lemma}\label{L:cra}
Suppose $F\in C^2(\R)$ and $u\in C^2(\RN)$. Then we have
\begin{equation}\label{crgrada}
\Phi(\nabla_z (F\circ u))^2 + \frac{\Phi^0(z)^{2\alpha}}4 \Psi(\nabla_\sigma (F\circ u))^2 = F'(u)^2 \left[\Phi(\nabla_z u)^2 + \frac{\Phi^0(z)^{2\alpha}}4 \Psi(\nabla_\sigma u)^2\right],
\end{equation}
and
\begin{equation}\label{cra}
\fBa(F\circ u) = F'(u) \fBa(u) + F''(u) \left[\Phi(\nabla_z u)^2 + \frac{\Phi^0(z)^{2\alpha}}4 \Psi(\nabla_\sigma u)^2\right].
\end{equation}
\end{lemma}
If we consider the function
\begin{equation}\label{F}
F(t) = t^{-(m+(\alpha+1)k-2)},
\end{equation}
then it is clear that 
\[
\begin{cases}
F'(t) = -(m+(\alpha+1)k-2) t^{-(m+(\alpha+1)k-1)},
\\
F''(t) = (m+(\alpha+1)k-2)(m+(\alpha+1)k-1) t^{-(m+(\alpha+1)k)},
\end{cases}
\]
and therefore $F(t)$ solves the ODE
\begin{equation}\label{ode}
F''(t) + \frac{m+(\alpha+1)k-1}t F'(t) = 0.
\end{equation}
Applying Lemma \ref{L:cra} to the function $F\circ \rho = \rho^{-(m+(\alpha+1)k-2)}$, we thus find
\begin{equation}\label{crav}
\fBa(\rho^{-(m+(\alpha+1)k-2)}) = F'(\rho) \fBa(\rho) + F''(\rho) \left[\Phi(\nabla_z \rho)^2 + \frac{\Phi^0(z)^{2\alpha}}4 \Psi(\nabla_\sigma \rho)^2\right].
\end{equation}
If we now define a function $K$ by the equation
\begin{equation}\label{rhoa}
K(z,\sigma) = \rho(z,\sigma)^{2(\alpha+1)},
\end{equation}
and again apply Lemma \ref{L:cra} with $F(t) = t^{\frac{1}{2(\alpha+1)}}$, we easily obtain the following.

\begin{lemma}\label{L:yam1a}
We have
\begin{equation}\label{gradrhoa}
\Phi(\nabla_z \rho)^2 + \frac{\Phi^0(z)^{2\alpha}}4 \Psi(\nabla_\sigma \rho)^2 = \frac{1}{4(\alpha+1)^2\rho^{4\alpha+2}} \left[\Phi(\nabla_z K)^2 + \frac{\Phi^0(z)^{2\alpha}}4 \Psi(\nabla_\sigma K)^2\right],
\end{equation}
and
\begin{equation}\label{yam2a}
\fBa(\rho) = \frac{1}{2(\alpha+1)\rho^{2\alpha+1}} \fBa(K) - \frac{2\alpha+1}{4(\alpha+1)^2 \rho^{4\alpha+3}} \left[\Phi(\nabla_z K)^2 + \frac{\Phi^0(z)^{2\alpha}}4 \Psi(\nabla_\sigma K)^2\right].
\end{equation}
\end{lemma}

To continue our discussion, we observe that substitution of \eqref{gradrhoa}, \eqref{yam2a} in \eqref{crav}, gives
\begin{align}\label{gummyy}
& \fBa(\rho^{-(m+(\alpha+1)k-2)}) 
\\
&= F'(\rho) \left\{\frac{1}{2(\alpha+1)\rho^{2\alpha+1}} \fBa(K) - \frac{2\alpha+1}{4(\alpha+1)^2 \rho^{4\alpha+3}} \left[\Phi(\nabla_z K)^2 + \frac{\Phi^0(z)^{2\alpha}}4 \Psi(\nabla_\sigma K)^2\right]\right\}
\notag\\
& + F''(\rho) \frac{1}{4(\alpha+1)^2\rho^{4\alpha+2}} \left[\Phi(\nabla_z K)^2 + \frac{\Phi^0(z)^{2\alpha}}4 \Psi(\nabla_\sigma K)^2\right]
\notag\\
& = \frac{F'(\rho) }{2(\alpha+1)\rho^{2\alpha+1}} \fBa(K) + \frac{1}{4(\alpha+1)^2\rho^{4\alpha+2}} \left(F''(\rho) - \frac{2\alpha+1}{\rho} F'(\rho) \right)\left[\Phi(\nabla_z K)^2 + \frac{\Phi^0(z)^{2\alpha}}4 \Psi(\nabla_\sigma K)^2\right]
\notag\\
& = \frac{1}{4(\alpha+1)^2\rho^{4\alpha+2}}\bigg\{2(\alpha+1)\frac{F'(\rho)}{\rho} \ K\ \fBa(K)
\notag\\
&+ \left(F''(\rho) - \frac{2\alpha+1}{\rho} F'(\rho) \right)\left[\Phi(\nabla_z K)^2 + \frac{\Phi^0(z)^{2\alpha}}4 \Psi(\nabla_\sigma K)^2\right]\bigg\}.
\notag
\end{align}

We have the following.

\begin{proposition}\label{P:magic}
Suppose that, for some $A>0$, there exist a function $K>0$ that satisfy the equation
\begin{equation}\label{magic}
2(\alpha+1) K\ \fBa(K) =  (m+(\alpha+1)k+2\alpha) \left[\Phi(\nabla_z K)^2 + \frac{\Phi^0(z)^{2\alpha}}4 \Psi(\nabla_\sigma K)^2\right] + A\ K^{\frac{2\alpha}{\alpha+1}}.
\end{equation}
Then the function $\rho>0$, connected to $K$ by \eqref{rhoa}, solves the intertwining equation \eqref{gummy}, with $\la = \frac{A(m+(\alpha+1)k-2)}{4(\alpha+1)^2}$.
\end{proposition}

\begin{proof}
Let $K>0$ be a solution to \eqref{magic}. With $F(t)$ as in \eqref{F}, we insert the equation \eqref{magic} in \eqref{gummyy}. Keeping in mind that \eqref{rhoa} gives $K^{\frac{2\alpha}{\alpha+1}} = \rho(z,\sigma)^{4\alpha}$ , we find that  the function $\rho$ satisfies the equation
\begin{align*}
& \fBa(\rho^{-(m+(\alpha+1)k-2)}) 
\\
& = \frac{1}{4(\alpha+1)^2\rho^{4\alpha+2}} \bigg\{\left(F''(\rho) + \frac{m+(\alpha+1)k-1}{\rho} F'(\rho)\right)\left[\Phi(\nabla_z K)^2 + \frac{\Phi^0(z)^{2\alpha}}4 \Psi(\nabla_\sigma K)^2\right]
\\
& + A F'(\rho) \rho^{4\alpha-1}\bigg\} 
\\
& =  \frac{A}{4(\alpha+1)^2} \frac{F'(\rho)}{\rho^{3}}
 =  -\frac{A(m+(\alpha+1)k-2)}{4(\alpha+1)^2} \rho^{-(m+(\alpha+1)k+2)},
\end{align*}
where in the second to the last equality we have used \eqref{ode} to kill the enemy, i.e., the term 
\[
\left(F''(\rho) + \frac{m+(\alpha+1)k-1}{\rho} F'(\rho)\right)\left[\Phi(\nabla_z K)^2 + \frac{\Phi^0(z)^{2\alpha}}4 \Psi(\nabla_\sigma K)^2\right].
\]
This proves that \eqref{gummy} holds with $\la = \frac{A(m+(\alpha+1)k-2)}{4(\alpha+1)^2}$.

\end{proof}

Summarising, we have proved that the key to solving the intertwining equation \eqref{gummy} is finding a function $K$ which satisfies \eqref{magic}. Once such $K$ is found, the function $\rho$ connected to $K$ by \eqref{rhoa}, will satisfy \eqref{gummy}. 

At this point, we specialise our discussion to the case $\alpha=1$, and show how to construct $K$ in such case.
We note explicitly that when $\alpha=1$ the equation \eqref{magic} reads
\begin{equation}\label{magicone}
4  \fB(K) =  \frac{m+2k+2}K \left[\Phi(\nabla_z K)^2 + \frac{\Phi^0(z)^{2}}4 \Psi(\nabla_\sigma K)^2\right] + A,
\end{equation}
where $\fB$ is as in \eqref{Fel2}.
We have the following result.

\begin{theorem}\label{T:main}
For $\ve>0$ we define a function $K:\RN\to (0,\infty)$ by the formula
\begin{equation}\label{rho}
K(z,\sigma) = (\ve^2 + \Phi^0(z)^2)^2 + 16 \Psi^0(\sigma)^2.
\end{equation}
Then $K$ solves the equation \eqref{magicone} with $A = 16 m \ve^2$.
\end{theorem}

The proof of Theorem \ref{T:main} will follow from the next two lemmas.

\begin{lemma}\label{L:yam3}
With $K$ as in \eqref{rho} of Theorem \ref{T:main}, one has
\[
\Phi(\nabla_z K)^2 + \frac{\Phi^0(z)^{2}}4 \Psi(\nabla_\sigma K)^2 = 16\ \Phi^0(z)^2\ K.
\]
\end{lemma}

\begin{proof}
For fixed $z\in \Rm$ and $\sigma\in \R^k$, consider the functions
\[
f(s) = (\ve^2 + s^2)^2 + 16 \Psi^0(\sigma)^2,\ \ \ \ \ \ \ g(t) = (\ve^2 + \Phi^0(z)^2)^2 + 16 t^2.
\]
Clearly, 
\[
K(z,\sigma) = f(\Phi^0(z)) = g(\Psi^0(\sigma)),
\]
and we have
\[
f'(s) = 4 s (\ve^2 + s^2),\ \ \ \ \  f''(s) = 4 (\ve^2 + s^2) + 8 s^2,\ \ \ \ \ g'(t) = 32 t,\ \ \ \ \ \  g''(t) = 32.
\]
From the chain rule one has
\[
\nabla_z K = f'(\Phi^0(z)) \nabla_z \Phi^0(z),\ \ \ \ \ \nabla_\sigma K = g'(\Psi^0(\sigma)) \nabla_\sigma \Psi^0(\sigma),
\]
and therefore the homogeneity of $\Phi$ and $\Psi$ implies
\[
\Phi(\nabla_z K)^2 = f'(\Phi^0(z))^2 \Phi(\nabla_z \Phi^0(z))^2 = 16 \Phi^0(z)^2 (\ve^2 + \Phi^0(z)^2)^2,\ \ \ \ \Psi(\nabla_\sigma K)^2 = 32^2 \Psi^0(\sigma)^2,
\]
where we have used the key identity \eqref{Finabla} from the Appendix.
These equations give
\begin{align*}
& \Phi(\nabla_z K)^2 + \frac{\Phi^0(z)^{2}}4 \Psi(\nabla_\sigma K)^2 = 16 \Phi^0(z)^2 (\ve^2 + \Phi^0(z)^2)^2 +  \frac{\Phi^0(z)^{2}}4 32^2 \Psi^0(\sigma)^2
\\
& = 16 \Phi^0(z)^{2} \left((\ve^2 + \Phi^0(z)^2)^2 + 16 \Psi^0(\sigma)^2\right) = 16\ \Phi^0(z)^{2}\ K.
\end{align*}

\end{proof}

\begin{lemma}\label{L:yam2}
We have 
\begin{equation}\label{delphiK}
\Delta_\Phi(K) = (4m+8) \Phi^0(z)^2 + 4 m \ve^2,\ \ \ \ \ \ \ \ \Delta_\Psi(K) = 32 k.
\end{equation}
As a consequence,
\begin{equation}\label{fBK}
\fB(K) = 4 (m+2k+2) \Phi^0(z)^2 + 4 m \ve^2.
\end{equation}
\end{lemma}

\begin{proof}
To prove \eqref{delphiK} we use the functions $f(s)$ and $g(t)$, introduced in the proof of Lemma \ref{L:yam3}, in combination with Proposition \ref{P:fin} below. Applying the latter twice (one time with $M = \Phi$, the other with $M = \Psi$), we find
\begin{align*}
\Delta_\Phi(K) & = f''(\Phi^0(z)) + \frac{m-1}{\Phi^0(z)} f'(\Phi^0(z)) 
\\
& = 4 (\ve^2 + \Phi^0(z)^2) + 8 \Phi^0(z)^2 + \frac{m-1}{\Phi^0(z)} 4 \Phi^0(z) (\ve^2 + \Phi^0(z)^2)
\\
& = (4m+8) \Phi^0(z)^2 + 4 m \ve^2.
\end{align*}
Similarly, we obtain
\begin{align*}
\Delta_\Psi(K) & = g''(\Psi^0(\sigma)) + \frac{k-1}{\Psi^0(\sigma)} g'(\Psi^0(\sigma)) 
 = 32  + \frac{k-1}{\Psi^0(\sigma)} 32 \Psi^0(\sigma) = 32 k.
\end{align*}
This establishes \eqref{delphiK}. The desired conclusion \eqref{fBK} immediately follows from this equation.

\end{proof}

We can now give the

\begin{proof}[Proof of Theorem \ref{T:main}]
By Lemma \ref{L:yam3} we have for the function $K$ 
\[
\frac{\Phi(\nabla_z K)^2 + \frac{\Phi^0(z)^{2}}4 \Psi(\nabla_\sigma K)^2}{K} = 16\ \Phi^0(z)^2.
\]
On the other hand, \eqref{fBK} gives
\[
4\fB(K) = 16 (m+2k+2) \Phi^0(z)^2 + 16 m \ve^2.
\]
Combining these two equations, it is clear that \eqref{magicone} does hold with $A = 16 m \ve^2$.

\end{proof}

\begin{proof}[Proof of Theorem \ref{T:yam}]
In Theorem \ref{T:main} we have proved that the function defined by $\rho(z,\sigma) = K(z,\sigma)^{1/4}$, where $K$ is defined by \eqref{rho}, satisfies the intertwining equation
\begin{equation}\label{yammyammy}
\fB(\rho^{-(2(k-1)+m)}) = - m(m + 2(k-1)) \ve^2 \rho^{-(2(k+1)+m)}.
\end{equation}
If we now keep in mind that \eqref{yam} gives when $\sigma_0 = 0\in \R^k$,
\[
K_{\ve,0} = ((m+2(k-1))\ve^2)^{\frac{m+2(k-1)}4} \rho^{-(2(k-1)+m)},
\]
from \eqref{yammyammy} and a rescaling argument based on the observation that $\fB(\la u) = \la \fB(u)$, we easily reach the desired conclusion that $K_{\ve,0}$ solves \eqref{yamme}. Since the equation is invariant under translations along $M = \{0\}_z\times \R^k$, it is clear that $K_{\ve,\sigma_0}$ is also a solution for every $\sigma_0\in \R^k$.

\end{proof}

%%%%%%%%%%%%%%%%%%%%%%%%%%%%%%%%%%%%%%%

%%%%%%%%%%%%%%%%%%%%%%%%%%%%%%%%%%%%%%%%%%%%%%%%%%%%%%

\section{Appendix: Minkowski norms}\label{S:app}

Let $M:\Rn\to [0,\infty)$ be a Minkowski norm in $\Rn$. By this we mean that $M^2\in \C^{2}(\Rn\setminus\{0\})$ is a strictly convex function such that $M(\la x) = |\la| M(x)$ for every $x\in \Rn$ and $\la\in \R$. Since all norms in $\Rn$ are equivalent, there exist constants $\beta \ge \alpha>0$ such that
\begin{equation}\label{equinorm}
\alpha |\xi| \le M(\xi) \le \beta |\xi|.
\end{equation} 
We denote by 
\begin{equation}\label{dual}
M^0(x) = \underset{M(\xi)=1}{\sup}\ \sa x,\xi\da,
\end{equation}
its Legendre transform, also known as the dual norm of $M$. The Cauchy-Schwarz inequality trivially holds 
\begin{equation*}
|\sa x,y\da| \le M(x) M^0(y).
\end{equation*}
A basic property of the norms $M$ and $M^0$ is the following, see \cite[Lemma 2.1]{BP}, and also (3.12) in \cite{CS},
\begin{equation}\label{Finabla}
M(\nabla M^0(x)) = M^0(\nabla M(x))  = 1,\ \ \ \ \ \ \ \ x\in \Rn\setminus\{0\}.
\end{equation}
Given a function $u\in C^1(\Rn)$, an elementary, yet useful,  consequence of the homogeneity of $M$ is
\begin{equation}\label{hom}
\sa \nabla M(\nabla u(x)),\nabla u(x)\da = M(\nabla u(x)).
\end{equation}
A less obvious basic fact is the following formula, which can be found in \cite[Lemma 2.2]{BP}.

\begin{lemma}\label{L:BP}
For every $x\in \Rn\setminus\{0\}$, one has
\[
M^0(x) \nabla M(\nabla M^0(x)) = x,\ \ \ \ \ \ M(x) \nabla M^0(\nabla M(x)) = x. 
\]
\end{lemma}

The Euler-Lagrange equation of the energy
\begin{equation}\label{finen}
\mathscr E_{M}(u) = \frac 12 \int M(\nabla u)^2 dx
\end{equation}
is the so-called Finsler Laplacian
\begin{equation}\label{EL}
\Delta_M(u) = \operatorname{div}(M(\nabla u)\nabla M(\nabla u)) = 0.
\end{equation}
It is worth emphasising here that the operator in \eqref{EL} is quasilinear, but not linear, unless of course $M(x) = |x|$. However, the operator $\Delta_M$ is elliptic. In fact, since $M^2$ is homogeneous of degree $2$, Euler formula gives
\[
\sa\nabla(M^2)(\xi),\xi\da = M(\xi)^2,
\]
and from \eqref{equinorm} we thus have for every $\xi\in \Rn$
\[
\alpha^2 |\xi|^2 \le \sa\nabla(M^2)(\xi),\xi \da \le \beta^2 |\xi|^2.
\]
An important property of the dual norm is the following result which follows from \eqref{Finabla} and from Lemma \ref{L:BP}, see \cite{CS} and \cite{FK}.

\begin{proposition}\label{P:fin}
Consider the function 
\[
\psi(x) = M^0(x).
\]
Then if $k\in C^2(\R)$, and $v = k\circ \psi$, one has in $\Rn\setminus\{0\}$
\[
\Delta_M(v) = k''(\psi) + \frac{n-1}\psi k'(\psi).
\]
\end{proposition}

%%%%%%%%%%%%%%%%%%%%%%%%%%%%%%%%%%%%%%%%%%%%%%%%%%%%%%%%%%%%%%%%%%%%%%%%%%

\bibliographystyle{amsplain}

\end{document}